\documentclass[letterpaper, 10pt, conference]{ieeeconf}      

\IEEEoverridecommandlockouts                              
\usepackage{amsmath}    
\usepackage{amsfonts}
\usepackage{graphicx}   
\usepackage{subcaption}
\usepackage{epsfig} 
\usepackage{color}
\usepackage[normalem]{ulem}
\usepackage{cancel}
\usepackage{amssymb}
\usepackage{color}
\usepackage[ruled,vlined,titlenotnumbered]{algorithm2e} 

\newcommand{\R}{\mathbb{R}}
\newcommand{\xset}{\mathcal{X}}
\newcommand{\yset}{\mathcal{Y}}
\newcommand{\xfset}{\mathbb{X}}
\newcommand{\yfset}{\mathbb{Y}}
\newcommand{\cset}{\mathcal{U}}
\newcommand{\cfset}{\mathbb{U}}
\newcommand{\dset}{\mathcal{D}}
\newcommand{\dfset}{\mathbb{D}}
\newcommand{\reachset}{\mathcal{V}}
\newcommand{\targetset}{\mathcal{L}}

\newtheorem{prop}{Proposition}

\title{\LARGE \bf Fast Reachable Set Approximations via State Decoupling Disturbances}

\author{Mo Chen*, Sylvia Herbert*, and Claire J. Tomlin
\thanks{This work has been supported in part by NSF under CPS:ActionWebs (CNS-931843), by ONR under the HUNT (N0014-08-0696) and SMARTS (N00014-09-1-1051) MURIs and by grant N00014-12-1-0609, by AFOSR under the CHASE MURI (FA9550-10-1-0567). The research of M. Chen has received funding from the ``NSERC PGS-D'' Program.}
\thanks{* Both authors contributed equally to this work. All authors are with the Department of Electrical Engineering and Computer Sciences, University of California, Berkeley. \{mochen72, sylvia.herbert, tomlin\}@berkeley.edu}
}

\begin{document}
\maketitle
\thispagestyle{empty}
\pagestyle{empty}

\begin{abstract}
With the recent surge of interest in using robotics and automation for civil purposes, providing safety and performance guarantees has become extremely important. In the past, differential games have been successfully used for the analysis of safety-critical systems. In particular, the Hamilton-Jacobi (HJ) formulation of differential games provides a flexible way to compute the reachable set, which can characterize the set of states which lead to either desirable or undesirable configurations, depending on the application. While HJ reachability is applicable to many small practical systems, the curse of dimensionality prevents the direct application of HJ reachability to many larger systems. To address computation complexity issues, various efficient computation methods in the literature have been developed for approximating or exactly computing the solution to HJ partial differential equations, but only when the system dynamics are of specific forms. In this paper, we propose a flexible method to trade off optimality with computation complexity in HJ reachability analysis. To achieve this, we propose to simplify system dynamics by treating state variables as disturbances. We prove that the resulting approximation is conservative in the desired direction, and demonstrate our method using a four-dimensional plane model.
\end{abstract}

\section{Introduction}
Optimal control problems and differential games are important and powerful theoretical tools for analyzing a wide variety of systems, particularly in safety-critical scenarios. They have been extensively studied in the past several decades \cite{Varaiya67, Evans84, Barron90, Tomlin00, Mitchell05, Bokanowski10, Fisac15}, and have been successfully applied to practical problems such as pairwise collision avoidance \cite{Mitchell05}, aircraft in-flight refueling \cite{Ding08}, vehicle platooning \cite{Chen15b}, and many others \cite{Bayen07, Huang11}. With the recent growing interest in using safety-critical autonomous systems such as autonomous cars and unmanned aerial vehicles for civil purposes \cite{FAA13, Amazon16, BBC16, AUVSI16, NASA16, Kopardekar16}, the importance and necessity of having flexible tools that can provide safety guarantees have substantially increased.

Intuitively, in an optimal control problem, one seeks to find the cheapest way a system described by an ordinary differential equation (ODE) model can perform a certain task. In a differential game, a system is controlled by two adversarial agents competing to respectively minimize and maximize a joint cost function. Hamilton-Jacobi (HJ) reachability is a common and effective way to analyze both optimal control problems and differential games because of the guarantees that it provides and its flexibility with respect to the system dynamics. 

In a reachability problem, one is given some system dynamics described by an ODE, and a target set which describes the set of final conditions under consideration. Depending on the application, the target set can represent either a set of desired or undesired states. The goal in reachability analysis is to compute the backward reachable set (BRS). When the target set is a set of desired states, the BRS represents the set of states from which the system can be guaranteed to be driven to the target set, despite the worst case disturbance. In contrast, when the target set is a set of undesired states, the BRS represents the set of states from which the system may be driven into the target set under some disturbance, despite its best control efforts to remain outside. Because of the theoretical guarantees that reachability analysis provides, it is ideal for analyzing the newest problems involving autonomous systems. 

Despite these advantages, there are serious drawbacks to using HJ Reachability on large systems. In order to compute the BRS, an HJ partial differential equation (PDE) must be solved on a grid representing a discretization of the state space. This means that the complexity of computing a BRS grows exponentially with the number of system states, making the standard HJ reachability formulation intractable for systems higher than approximately five dimensions. To address this difficulty, a number of approximation techniques have been developed, such as those involving projections, approximate dynamic programming, and occupation measure theory \cite{Mitchell03, McGrew08, Lasserre08}. 

In this paper, we propose a general method to remove coupling in systems by treating coupling variables as disturbances. This uncoupling of dynamics transforms the system into a form that is suitable for analysis using methods such as \cite{Mitchell11} and \cite{Chen15}, which exploit system structure. This method can also be combined with previous work such as \cite{Mitchell03, McGrew08, Lasserre08}. to reduce computation complexity even further. We show that our approach results in BRSs that are conservative in the desired direction, and demonstrate the performance of our method when combined with the decoupled formulation in \cite{Chen15}.

\section{Problem Formulation \label{sec:formulation}}
Consider a differential game between two players described by the time-invariant system with state $z\in \R^n$ evolving according to the ODE

\begin{equation}
\label{eq:dyn}
\begin{aligned}
\dot{z} &= f(z, u, d), \quad t \in [-T, 0] \\
u &\in \cset, d \in \dset
\end{aligned}
\end{equation}

\noindent where $u$ is the control of Player 1, and $d$ is the control of Player 2. Often, $u$ and $d$ are regarded as the control and disturbance, respectively, of a system described by \eqref{eq:dyn}. We assume $f:\R^n\times \cset \times \dset \rightarrow \R^n$ is uniformly continuous, bounded, and Lipschitz continuous in $z$ for fixed $u,d$, and the control functions $u(\cdot)\in\mathbb{U},d(\cdot)\in\mathbb{D}$ are drawn from the set of measurable functions\footnote{A function $f:X\to Y$ between two measurable spaces $(X,\Sigma_X)$ and $(Y,\Sigma_Y)$ is said to be measurable if the preimage of a measurable set in $Y$ is a measurable set in $X$, that is: $\forall V\in\Sigma_Y, f^{-1}(V)\in\Sigma_X$, with $\Sigma_X,\Sigma_Y$ $\sigma$-algebras on $X$,$Y$.}. As in \cite{Varaiya67, Evans84, Mitchell05}, we allow Player 2 to only use nonanticipative strategies $\gamma$, defined by

\begin{equation}
\begin{aligned}
\gamma \in \Gamma &:= \{\mathcal{N}: \cfset \rightarrow \dfset \mid u(r) = \hat{u}(r) \text{ a. e. } r\in[t,s] \\
&\Rightarrow \mathcal{N}[u](r) = \mathcal{N}[\hat{u}](r) \text{ a. e. } r\in[t,s]\}
\end{aligned}
\end{equation}

Suppose that the state can be written as $z = (x,y)$ such that the control $u$ and disturbance $d$ can be written as $u = (u_x, u_y), d = (d_x, d_y)$, and such a decomposition of the control leads to the following form of system dynamics:
\begin{equation}
\label{eq:scdyn} 
\begin{aligned}
\dot{x} &= g(x, y, u_x, d_x) \quad \dot{y} = h(x, y, u_y, d_y) \\
u_x &\in \cset_x, u_y \in \cset_y, \quad d_x \in \dset_x, d_y \in \dset_y, \quad t \in [-T, 0] \\
\end{aligned}
\end{equation}

\noindent where $x \in \R^{n_x}, y \in \R^{n_y}, n_x + n_y = n$, and $g,h$ are components of the system dynamics that involve $(u_x, d_x), (u_y, d_y)$, respectively. Note that this assumption on $u$ and $d$ is a mild one, and is satisfied by any system in which each of the control components $u_x, u_y$ and disturbance components $d_x, d_y$ have independent control sets $\cset_x, \cset_y$ and disturbance sets $\dset_x, \dset_y$, respectively; note that we can also write $\cset = \cset_x \times \cset_y$, and $\dset = \dset_x \times \dset_y$. This decomposition is very common in real world systems, where control input bounds such as maximum acceleration and maximum turn rate are independent of each other.

For the system in the form of \eqref{eq:scdyn}, we would like to compute the BRS of time horizon $T$, denoted $\reachset(T)$. Intuitively, $\reachset(T)$ is the set of states from which there exists a control strategy to drive the system into a target set $\targetset$ within a duration of $T$ despite worst-case disturbances. Formally, the BRS is defined as\footnote{Similar definitions of BRSs and their relationships can be found in, for example, \cite{Mitchell07}}

\begin{equation}
\label{eq:reachset}
\begin{aligned}
\reachset(T) = &\{z_0\in\R^n: \exists u(\cdot)\in\cfset, \forall \gamma\in\Gamma, \\
&z(\cdot) \text{ satisfies \eqref{eq:scdyn}}, z(-T) = z_0\Rightarrow z(0) \in \targetset\}
\end{aligned}
\end{equation}

Standard HJ formulations exist for computing the BRS in the full dimensionality $n$ \cite{Barron90, Mitchell05, Bokanowski10, Fisac15}. In addition, special HJ formulations can be used to substantially reduce computation complexity for systems with special properties such as having terminal integrators or having fully decoupled dynamics \cite{Mitchell11, Chen15}. The goal of this paper will be to demonstrate how to take advantage of previous work on BRS computation for systems of particular forms, even when the actual system dynamics do not exactly satisfy necessary assumptions. For concreteness, in this paper we will focus on removing coupling to put systems into a fully decoupled form that satisfies the assumptions in \cite{Chen15}.

Our proposed approach computes an approximation of the BRS in dimension $\max(n_x, n_y)$ instead of in dimension $n$, dramatically reducing computation complexity. This is done by removing coupling in the dynamics by treating certain variables as disturbances. The computed approximation is conservative in the desired direction, meaning any state in the approximate BRS is also in the true BRS. 

\section{Background \label{sec:summary}}
Given the system in \eqref{eq:dyn}, the BRS $\reachset(t)$ defined in \eqref{eq:reachset} can be computed using the following terminal value HJ PDE:
\vspace{-0.5em}
\begin{equation}
\label{eq:HJIPDE}
\begin{aligned}
D_t V(t, z) + \min_{u\in\cset} \max_{d\in\dset} D_z V(t,z) \cdot f(z,u,d) &= 0 \\
V(0,z) &= l(z)
\end{aligned}
\end{equation}

In \eqref{eq:HJIPDE}, the target set $\targetset$ is represented as the sub-zero level set of the function $l(z)$: $\targetset = \{z \in \R^n: l(z) \le 0\}$. The BRS $\reachset(T)$ is represented as the viscosity solution of \eqref{eq:HJIPDE} $V(t,z)$: $\reachset(t) = \{z \in \R^n: V(t, z) \le 0\}$ \cite{Mitchell05, Crandall83}. The definition given in \eqref{eq:reachset} assumes that the control $u$ tries to grow the BRS as much as possible, and the disturbance $d$ tries to do the opposite. Similar definitions of the BRS, for example one in which the control tries to inhibit its growth, can be computed by adjusting the minimum and maximum. For simplicity, we will assume that the BRS is defined by $\eqref{eq:reachset}$, and thus determined by solving $V(t, z)$ given in \eqref{eq:HJIPDE}.

In the case that the system is fully decoupled in the form of \eqref{eq:fddyn}, \cite{Chen15} provides a method to exactly compute $V(t, z)$. The computation is done in each of the decoupled components, substantially reducing computation complexity as long as the system is decoupled and as long as the target $\targetset$ can be written as the intersection $\cap_i \targetset_i$.

\begin{equation}
\label{eq:fddyn} 
\begin{aligned}
\dot{x_i} &= f_i(x_i, u_i, d_i), \quad t\in [-T,0] \\
u_i &\in \cset_i, d_i \in \dset_i, \quad i = 1,\ldots, N
\end{aligned}
\end{equation}

The reason for the dimension reduction is intuitive, and can be easily verified by checking that the following statements are equivalent:

\begin{enumerate}
\item $\exists u_i \in \cfset_i, \forall d_i \in \dfset_i, x_i(\cdot) \text{ satisfies \eqref{eq:fddyn}},\\x_i(0) \in \targetset_i, i = 1,2$
\item $\exists (u_1,u_2) \in \cfset_1 \times \cfset_2, \forall (d_1, d_2) \in \dfset_1 \times \dfset_2, \\ x_i(\cdot) \text{ satisfies \eqref{eq:fddyn}}, (x_1(0), x_2(0)) \in \targetset_1 \cap \targetset_2$
\item $\exists u \in \cfset, \forall d \in \dfset, x(\cdot) \text{ satisfies \eqref{eq:dyn}}, x(0) \in \targetset$
\end{enumerate}
%
%
%
%
%

\section{Coupling as Disturbance}
The decoupled formulation summarized in Section \ref{sec:summary} enables the exact computation of $n$ dimensional BRSs in the space of the $n_i$ dimensional decoupled components. However, this substantial computation benefit can only be gained if the system dynamics satisfy \eqref{eq:fddyn}.

In the case where the dynamics are in the form of \eqref{eq:scdyn}, the decoupled formulation cannot be directly used. However, if one treats $y$ as a disturbance in the function $g$, and $x$ as a disturbance in the function $h$, the system would become decoupled. Mathematically, \eqref{eq:scdyn} becomes the following:
\begin{equation}
\label{eq:udyn} 
\begin{aligned}[c]
\dot{x} &= \hat{g}(x, y_d, u_x, d_x)\\
x &\in \xset, y_d \in \yset\\
u_x &\in \cset_x, d_x \in \dset_x\\
\end{aligned}
\quad
\begin{aligned}[c]
\dot{y} &= \hat{h}(y, x_d, u_y, d_y) \\
x_d &\in \xset, y \in \yset\\
u_y &\in \cset_y, d_y \in \dset_y\\
\end{aligned}
\quad
\begin{aligned}[c]
\\
t &\in [-T, 0]\\
\\
\end{aligned}
\end{equation}

\noindent where $\xset \times \yset$ represents the full-dimensional domain over which computation is done. By treating the coupled variables as a disturbance, we have uncoupled the original system dynamics \eqref{eq:scdyn}, and produced approximate dynamics \eqref{eq:udyn} that are decoupled, allowing us to do the computation in the space of each decoupled component. 



Compared to the original system dynamics given in \eqref{eq:scdyn}, the uncoupled dynamics given in \eqref{eq:udyn} experiences a larger disturbance, since the $y$ dependence of the function $g$ and the $x$ dependence on the function $h$ are treated as disturbances. With the definition of BRS in \eqref{eq:reachset}, the approximate BRS computed using the dynamics \eqref{eq:udyn} is an under-approximation of the true BRS. We formalize this in the proposition below.

\begin{prop}
Let $\reachset_x(T), \reachset_y(T)$ be the BRSs of the subsystem \eqref{eq:udyn} from the target sets $\targetset_x, \targetset_y$, and let $\reachset(T)$ be the BRS of the system \eqref{eq:scdyn} from the target set $\targetset$. Then\footnote{Strictly speaking,  $\targetset_x, \targetset_y ,\reachset_x(T), \reachset_y(T)$ would need to be ``back projected'' into the higher dimensional space before their intersections can be taken, but we will use the abuse of notation for convenience.}, $\targetset = \targetset_x \cap \targetset_y \Rightarrow \reachset_x(T) \cap \reachset_y(T) \subseteq \reachset(T)$.
\end{prop}

\begin{proof}
It suffices to show that given any state $(x_0, y_0) = (x(-T), y(-T))$ such that $x_0, y_0$ are in the BRSs $\reachset_x(T), \reachset_y(T)$ for the system in \eqref{eq:udyn}, respectively, then $(x_0, y_0)$ is in the BRS $\reachset(T)$ for the system in \eqref{eq:scdyn}.

For convenience, we will use $x(\cdot) \in \xfset$ to denote $x(s) \in \xset~ \forall s \in [-T, 0]$, with $y(\cdot) \in \yfset$ having the analogous meaning. Applying the definition of BRS in \eqref{eq:reachset} to the subsystems in \eqref{eq:udyn}, at the state $z_0 = (x_0, y_0)$ we have
\begin{enumerate}
\item $\exists u_x \in \cfset_x, \forall d_x \in \dfset_x, \forall y_d \in \yfset, x(\cdot) \text{ satisfies \eqref{eq:udyn}},\\x(0) \in \targetset_x$
\item $\exists u_y \in \cfset_y, \forall d_y \in \dfset_y, \forall x_d \in \xfset, y(\cdot) \text{ satisfies \eqref{eq:udyn}},\\y(0) \in \targetset_y$
\end{enumerate}

The above two conditions together imply
\vspace{-0.5em}
\begin{equation}
\begin{aligned}
\exists (u_x, u_y) \in \cfset_x \times \cfset_y, \forall (d_x, d_y) \in \dfset_x \times \dfset_y, \\
\forall (x_d, y_d) \in \xfset \times \yfset, (x(\cdot), y(\cdot)) \text{ satisfies \eqref{eq:udyn}}, \\
(x(0), y(0)) \in \targetset
\end{aligned}
\end{equation}

In particular, since $x(\cdot) \in \xfset, y(\cdot) \in \yfset$, the above is true also when $x_d = x(\cdot), y_d = y(\cdot)$, so
\vspace{-0.5em}
\begin{equation}
\begin{aligned}
\exists (u_x, u_y) \in \cfset_x \times \cfset_y, \forall (d_x, d_y) \in \dfset_x \times \dfset_y, \\
(x_d, y_d) = (x(\cdot), y(\cdot)), (x(\cdot), y(\cdot)) \text{ satisfies \eqref{eq:udyn}}, \\
(x(0), y(0)) \in \targetset
\end{aligned}
\end{equation}

But if  $x_d = x(\cdot), y_d = y(\cdot)$, then \eqref{eq:udyn} becomes \eqref{eq:scdyn}, thus
\vspace{-0.5em}
\begin{equation}
\begin{aligned}
\exists (u_x, u_y) \in \cfset_x \times \cfset_y, \forall (d_x, d_y) \in \dfset_x \times \dfset_y, \\
(x(\cdot), y(\cdot)) \text{ satisfies \eqref{eq:scdyn}}, (x(0), y(0)) \in \targetset.
\end{aligned}
\end{equation}
\end{proof}

By treating the coupled states as disturbance, the computation complexity reduces from $O(k^nT)$ for the full formulation to $O(k^{\max\{n_x, n_y\}}T)$ for the decoupled approximate system \eqref{eq:udyn}.

\section{Disturbance Splitting \label{sec:couplingsplitting}}
By treating coupling variables $y$ and $x$ as disturbances in $g$ and $h$, respectively, we introduce conservatism in the BRS computation. This conservatism is always in the desired direction. In situations where $\xset$ and $\yset$ are large, the degree of conservatism can be reduced by splitting the disturbance $x_d$ and $y_d$ into multiple sections, as long as the target set $\targetset$ does not depend on the state variables being split. For example, $x_d \in \xset$ can be split as follows:

\begin{equation}
\begin{aligned}
x_d^i \in \xset_i, i = 1, 2, \ldots, M \\
\text{where } \bigcup_{i=1}^M \xset_i = \xset
\end{aligned}
\end{equation}

This disturbance splitting results in the following family of approximate system dynamics
\begin{equation}
\label{eq:usdyn} 
\begin{aligned}[c]
\dot{x} &= g(x, y_d, u_x, d_x) \\
u_x &\in \cset_x,d_x \in \dset_x \\
&y_d \in \yset \\
t &\in [-T, 0] \\
\end{aligned}
\qquad
\begin{aligned}[c]
\dot{y} &= h(y, x_d^i, u_y, d_y) \\
u_y &\in \cset_y,  d_y \in \dset_y \\
&x_d^i \in \xset^i \\
i &= 1, 2, \ldots, M_x \\
\end{aligned}
\end{equation}

\noindent from which a BRS can be computed in $\xset^i \times \yset$. Since $\xset^i \subseteq \xset$, the uncoupling disturbance is reduced whenever the disturbance $x_d$ is split.  In addition, the uncoupling disturbance $y_d$ can also be split into $M_y$, for a total of $M = M_x M_y$ total ``pieces'' of uncoupling disturbances. However, a smaller disturbance bound also restricts the allowable trajectories of each approximate system, so overall it is difficult to determine \textit{a priori} the optimal way to split the uncoupling disturbances. The trade-off between the size of disturbance bound and degree of restriction placed on trajectories can be seen in Fig. \ref{fig:VolRatio_V}.

\subsection{Examples of Decoupling System Dynamics}
Our proposed method applies to any system of the form \eqref{eq:scdyn}, as we will demonstrate with the example in Section \ref{sec:results}. Systems with light coupling between groups of state variables are particularly suitable for the application or our proposed method. Below are other example systems for which treating the coupling variables $y$ in $g$ or $x$ in $h$ as disturbances would lead to decoupling.

\textbf{Linear systems with large Jordan blocks}, for example,
\begin{equation}
\dot{z} = 
\begin{bmatrix}
1 & 1 & 0 & 0 \\
0 & 1 & 1 & 0 \\
0 & 0 & 1 & 1 \\
0 & 0 & 0 & 1
\end{bmatrix}
z + B u
\end{equation}

If $z_3$ is treated as a disturbance in the equation for $\dot{z}_2$, we would obtain the decoupled components $(z_1, z_2)$ and $(z_3, z_4)$.

%
%

\textbf{Lateral Quadrotor Dynamics near Hover} \cite{Bouffard12}:
\begin{equation}
\label{eq:quad_dyn}
\dot{z} = 
\begin{bmatrix}
z_2 \\
g \tan z_3 \\
z_4 \\
-d_0 z_3 - d_1 z_4
\end{bmatrix}
+
\begin{bmatrix}
0 \\
0 \\
0 \\
n_0
\end{bmatrix}
u
\end{equation}

\noindent where $z_1$ is the longitudinal position, $z_2$ is the longitudinal velocity, $z_3$ is the pitch angle, $z_4$ is the pitch rate, and $u$ is the desired pitch angle. These dynamics are valid for small pitch angles. The system would become decoupled into the $(z_1, z_2)$ and $(z_3, z_4)$ components if $z_3$ is treated as a disturbance in $\dot{z}_2$. In fact, the full ten-dimensional (10D) quadrotor dynamics given in \cite{Bouffard12} can be decomposed into five decoupled components of 2D systems, allowing an approximation of the 10D BRS to be computed in 2D space.

It is worth noting that after decoupling the 4D system in \eqref{eq:quad_dyn} into two 2D systems, each decoupled component is in the form $\dot{y} = g(y, u), \dot{x} = b(y)$. This is exactly the form of the dynamics in \cite{Mitchell11}, allowing the 4D BRS to be exactly computed in 1D! In general, removing coupling may bring the system dynamics into a form suitable for analysis using methods that require specific assumptions on the dynamics, potentially greatly reducing computation complexity.

\section{Numerical Results \label{sec:results}}

We demonstrate our proposed method using a 4D model of an aircraft flying at constant altitude, given by
\begin{equation}
	\begin{aligned}[c]
		\dot{p}_x &= v \cos \psi \\
		\dot{\psi} &= \omega \\
		\underline \omega &\le \omega \le \bar\omega
	\end{aligned}
	\qquad
	\begin{aligned}[c]
	\dot{p}_y &= v \sin \psi \\
	\dot{v} &= a \\
	 \underline  a & \le a \le \bar a
	\end{aligned}
\end{equation}

\noindent where $(p_x, p_y)$ represent the plane's position in the $x$ and $y$ directions, $\psi$ represents the plane's heading, and $v$ represents the plane's longitudinal velocity. The plane has a limited turn rate $\omega$ and a limited longitudinal acceleration $a$ as its control variables. For our example, the computation bounds are

$$
	\begin{aligned}[c]
	p_x &\in [-40, 40] \text{ m }\\
	\psi &\in [-\pi, \pi] \text{ rad}
	\end{aligned}
	\qquad
	\begin{aligned}[c]
	p_y &\in [-40, 40] \text{ m }\\
	v &\in [6, 12] \text{ m/s}
	\end{aligned}
$$

Using the decoupled approximation technique, we create the following decoupled approximation of the system with $(p_x, \psi)$ and $(p_y, v)$ as the decoupled components:
\begin{equation}
\begin{aligned}[c]
\dot{p}_x &= d_v \sin(\psi) \\
\dot{\psi} &= \omega \\
\underline \omega &\le \omega \le \bar\omega\\
\underline{d_v} &\leq d_v \leq \bar d_v 
\end{aligned}
\qquad
\begin{aligned}[c]
\dot{p}_y &= v \sin(d_{\psi})\\
 \dot{v} &= a\\
 \underline a &\le a \le \bar a\\
\underline{d_{\psi}} &\le d_{\psi} \le \bar{d_{\psi}}
\end{aligned}
\end{equation}

We define the target set as a square of length 4 m centered at $(p_x, p_y)=(0,0)$, described by $\targetset = \{(p_x, p_y, \psi, v): |p_x|, |p_y| \le 2\}$. This can be interpreted as a positional goal centered at the origin that can be achieved for all angles and velocities within the computation grid bounds. From the target set, we define $l(z)$ such that $l(z)\le 0 \Leftrightarrow x\in\mathcal{L}$. To analyze our newly decoupled system we must likewise decouple the target set by letting $\mathcal{L}_i, i = 1,2$ be 
\begin{equation}
\begin{aligned}
\mathcal{L}_1 &= \{(p_x, \psi): |p_x|\le 2\} \\
\mathcal{L}_2 &= \{(p_y, v): |p_y|\le 2\}
\end{aligned}
\end{equation}

These target sets have corresponding implicit surface functions $l_i(x_i),i=1,2$, which then form the 4D target set represented by $l(z) = \max_i l_i(x_i),i=1,2$.

We set $\mathcal{L}$ as the target set in our reachability problem and computed the BRS $\mathcal{V}(T)$ from $\mathcal{L}$ using both the direct 4D computation as well as the proposed decoupled approximation method.

\subsection{Backward Reachable Sets}
The BRS describes in this case the set of initial conditions from which the system is guaranteed to reach the target set within a given time period $T$ despite worst possible disturbances. To analyze the BRS we vary the degree of conservatism using the disturbance splitting method described in section \ref{sec:couplingsplitting}. After applying splitting, we arrive at the following piece-wise system:
\begin{equation}
\label{eq:System_Split}
\begin{aligned}[c]
\dot{p}_x &= d_v^i \sin(\psi) \\
\dot{\psi} &= \omega \\
\underline \omega &\le \omega \le \bar\omega\\
\underline d_v^i &\leq d_v^i \leq \bar d_v^i \\
i = &1,2,\dots,M_v 
\end{aligned}
\qquad
\begin{aligned}[c]
 \dot{p}_y &= v \sin(d_{\psi}^j)\\
 \dot{v} &= a\\
 \underline a &\le a \le \bar a\\
 \underline d_\psi^j &\le d_{\psi}^j \le \bar d_\psi^j\\
 j = &1,2,\ldots,M_\psi\\
\end{aligned}
\end{equation}

We analyzed the decoupled approximation formulation with $M_v$ and $M_\psi$ ranging from $1$ to $32$. To visually depict the computed 4D BRS, we plot 3D slices of the BRS in Fig. \ref{fig:3Dplots}. In these slices the green sets are the BRS computed using the full formulation, and the gray sets are the decoupled approximations. With the definition of BRS given in \eqref{eq:reachset}, all the decoupled approximations are constructed to be under-approximations.

\begin{figure*}
	\centering
	\includegraphics[width=\textwidth]{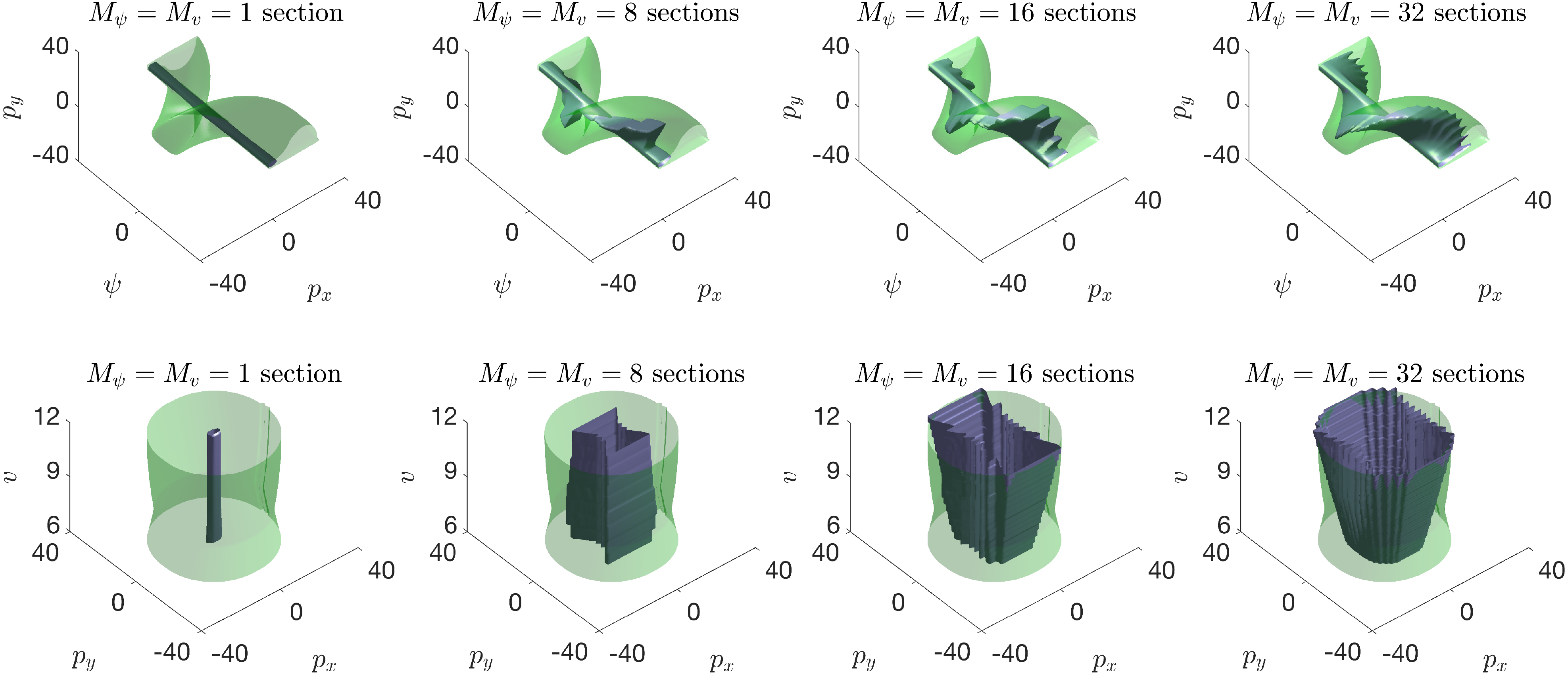}
	\caption{3D slices of the BRSs across a range of $M_v$ and $M_\psi$. The full formulation sets are in green, and the decoupled approximation sets are in gray. The top row shows 3D projections through all values of $v$. The top left plot shows this projection for $M_\psi = M_v = 1$ with $M_\psi, M_v$ increasing as we move right in the list of plots, up to $M_\psi = M_v = 32$ at top right. The bottom figures show the same sections for 3D projections through $\psi$.}
	\label{fig:3Dplots}
	\vspace{-.1in}
\end{figure*}

The top row of plots shows the 3D projections through all values of $v$. The bottom row of plots shows the 3D projections through all values of $\psi$. Moving from left to right, each column of plots shows the decoupled approximations with an increasing number of split sections $M_\psi$ and $M_v$. 

\subsection{Reconstruction Performance \label{sec:volumerecon}}
To compare the degree of conservatism of the decoupled approximations, we determined the total 4D volume of the BRS computed using both methods. We then took the ratio of the decoupled approximation volume to the full formulation BRS volume. Since under-approximations are computed, a higher volume ratio indicates a lower degree of conservatism. 

Fig. \ref{fig:VolRatio_V} shows this volume ratio as a function of $M_\psi$ and $M_v$. For example, the purple curve represents the volume ratio for $M_v = 1$ across various values of $M_\psi$, and on the other extreme, the yellow curve represents the the volume ratio for $M_v = 32$ across various values of $M_\psi$. The highest number of sections computed was with $M_\psi = M_v = 32$. 

\begin{figure}
	\centering
	\includegraphics[width=1\linewidth]{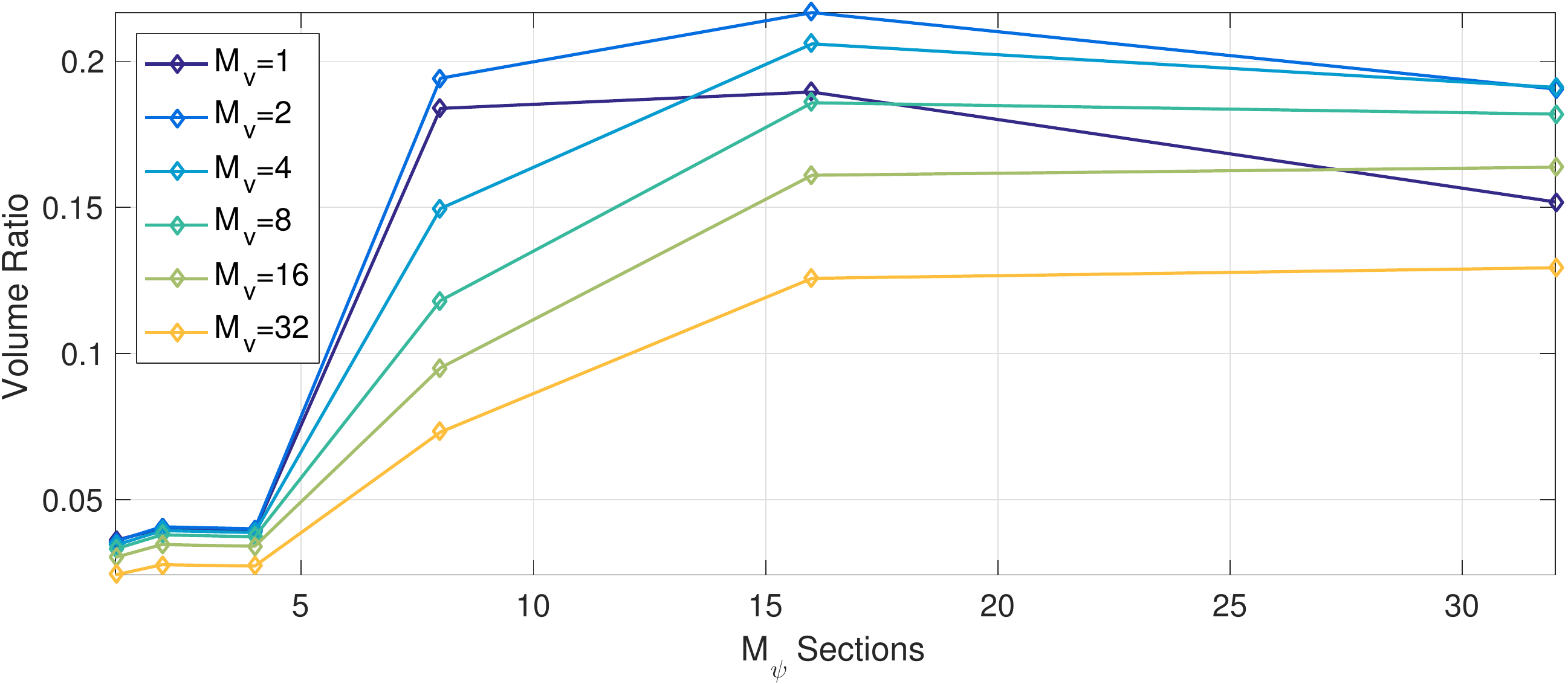}
	\caption{The vertical axis represents the ratio of the reconstructed BRS volume over the full formulation volume. The graph shows how this ratio changes as a function of number of disturbance sections. The highest volume ratio (and therefore least conservative BRS) was for $M_\psi = 16, M_v =2$.}
	\label{fig:VolRatio_V}
	\vspace{-.2in}
\end{figure}

Initially the decoupled approximations become less conservative as $M_v$ and $M_\psi$ increase. This is because splitting the disturbance range has the effect of mitigating the strength of the disturbances. However, splitting the disturbance range also restricts the allowable trajectories of the system and can eventually introduce more conservatism. For example, if the velocity disturbance range is $6 \le d_v \le 12$, the trajectories must stay within this velocity range for the duration of $T$. If this range is split, the set of disturbances has a smaller range, but likewise the trajectories for each subsection must remain within the smaller split velocity range for the time period. Therefore, there is an optimal point past which splitting does not help decrease conservatism.

In this system the least conservative approximation was for $M_\psi=16, M_v=2$. The volume ratio for this approximation was 0.217, meaning that the decoupled approximation had a volume that is 21.7\% of the volume of the BRS computed using the full formulation. The 3D projections of the set computed by $M_\psi=16, M_v=2$ can be seen in Fig. \ref{fig:Opt3DViews}.

\begin{figure}
\centering
\includegraphics[width=1\linewidth]{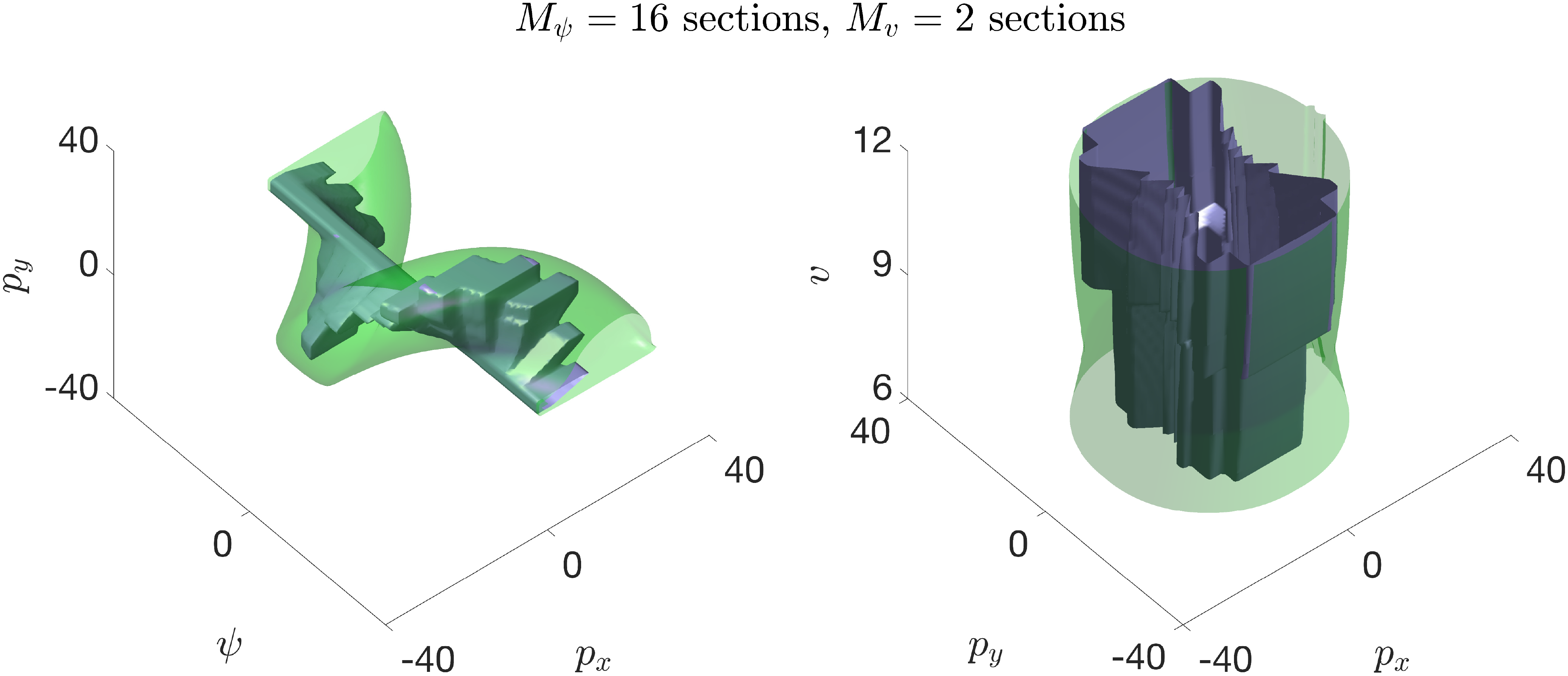}
\caption{3D slices of the BRS for $M_\psi = 16, M_v =2$. This decoupled approximation provides the largest and least conservative under-approximation.}
\label{fig:Opt3DViews}
\vspace{-.05in}
\end{figure}

\begin{figure}
	\centering
	\includegraphics[width=1\linewidth]{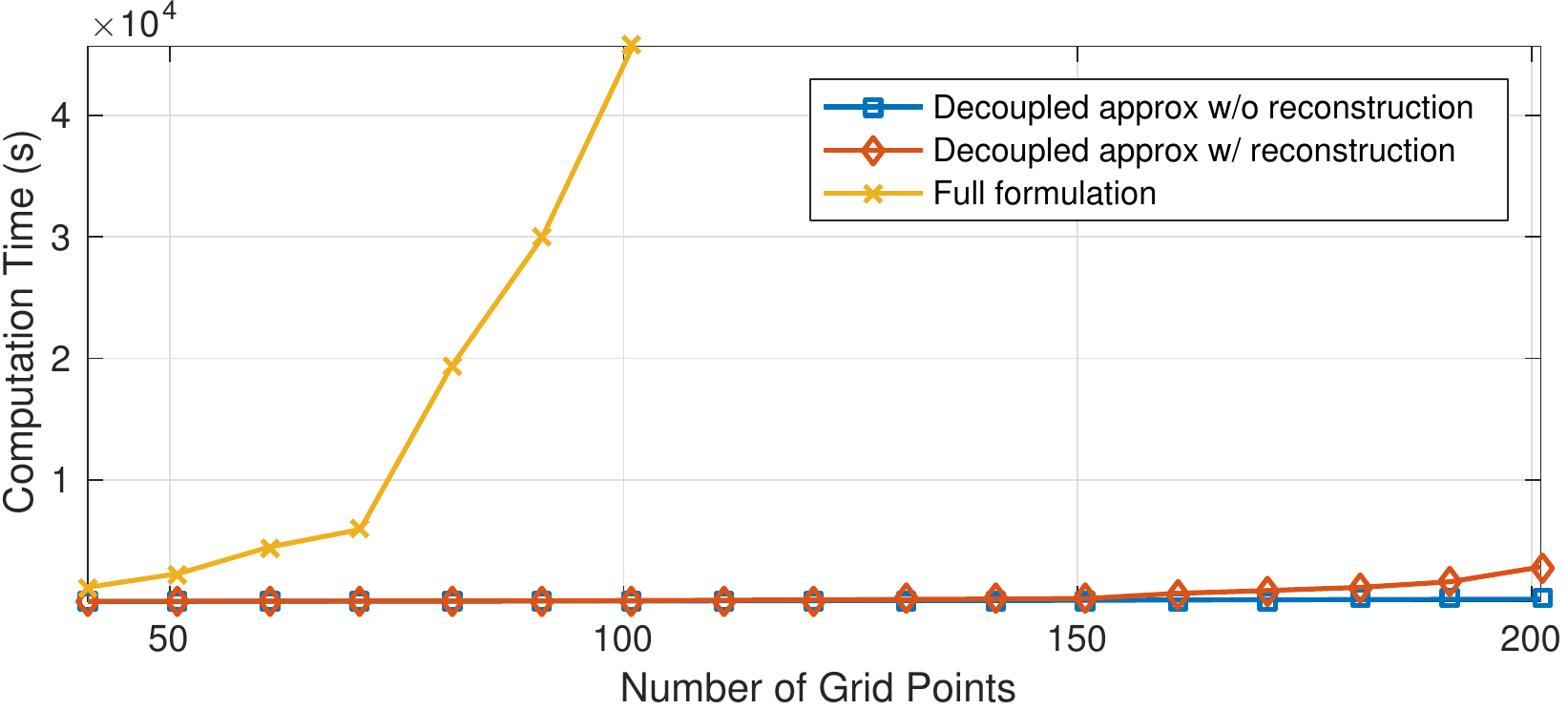}
	\caption{Computation time as a function of the number of grid points in each dimension. The full formulation (yellow curve) is orders of magnitude slower than the decoupled approximation. The decoupled approximation with reconstruction (red curve) takes a bit more time (and significantly more memory) than without reconstruction (blue curve).}
	\label{fig:CompTime_N}
	\vspace{-.2in}
\end{figure}

\subsection{Computation Time Performance}
In Fig. \ref{fig:CompTime_N} we compare the computation time of the two methods as a function of the number of grid points in each dimension. Computations were done on a desktop computer with a Core i7-5820K CPU and 128 GB of random-access memory (RAM). The full formulation (yellow curve) quickly becomes intractable as grid points are added; 100 grid points in each dimension requires 12.7 hours and 97 GB of RAM.

The decoupled approximation is orders of magnitude faster than the full formulation, and therefore can be done with many more grid points in each dimension. The decoupled approximation used was for $M_\psi=16, M_v =2$ sections, as this provided the most accurate BRS as determined in \ref{sec:volumerecon}. The runtimes would be even faster using $M_\psi = M_v = 1$ sections. We plot both the runtimes for reachability computation with 4D-reconstruction (red curve) as well as the runtimes for reachability computation alone (blue curve). Compared to the full formulation, at 100 grid points in each dimension the decoupled approximation takes 50 seconds to run and 36 seconds to reconstruct, with 625 MB of RAM to run and 6.75 GB of RAM to reconstruct. At 200 grid points the decoupled approximation takes 3.37 minutes to run and 44.1 minutes to reconstruct, with 1 GB of RAM to run and 120 GB of RAM to reconstruct.
 
In general we recommend computing the value function in only a region near a state of interest, bypassing the time and RAM required to reconstruct the function over the entire grid. Without full reconstruction of the value function, we are able to obtain results faster and for higher numbers of grid points before running out of memory, improving the accuracy of the computation.

In Fig. \ref{fig:CompTime_VSec} we compare the computation time of the 2D computations in the decoupled approximation as a function $M_v$ and $M_\psi$. Each line of the graph represents the computation time for a fixed number of $M_v$ across various values of $M_\psi$. As the number of sections increases, the computation time required increases approximately linearly, as expected.

\begin{figure}
	\centering
	\includegraphics[width=0.45\textwidth]{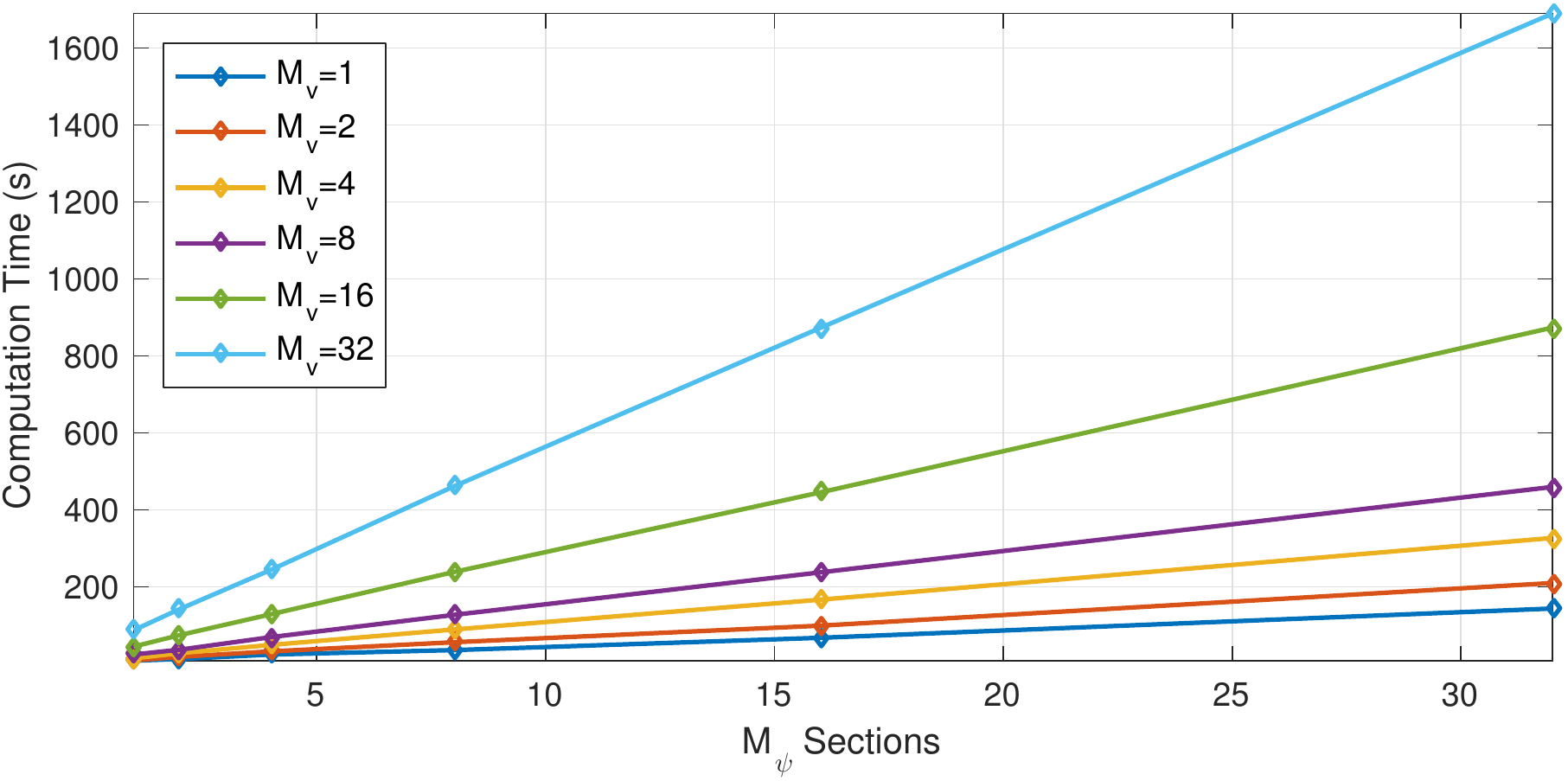}
	\caption{Computation time to compute decoupled approximation sets without reconstruction as a function of $M_v$ and $M_\psi$. As the number of sections increases, the computation time increases approximately linearly.}
	\label{fig:CompTime_VSec}
	\vspace{-.2in}
\end{figure}


\section{Conclusions}
Hamilton-Jacobi reachability analysis can provide safety and performance guarantees for many practical systems, but the curse of dimensionality limits its application to systems with less than approximately five state variables. By treating state variables as disturbances, key state dependencies can be eliminated, reducing the system dynamics to a simpler form and allowing reachable sets to be calculated conservatively using available efficient methods in the literature. 


\bibliographystyle{IEEEtran}
\bibliography{references}
\end{document}